\documentclass[11pt]{amsart}

\issueinfo{00}{0}{Month}{Year}
\copyrightinfo{2008}{University of Calgary}
\pagespan{1}{2}
\PII{ISSN 1715-0868}
\include{cdmstdcmds}

\usepackage{geometry}        
\usepackage{graphicx}
\usepackage{amssymb}
\usepackage{epstopdf}
\usepackage{amsfonts}
\usepackage{amsmath}
\usepackage{graphicx}
\usepackage{mathrsfs}
\usepackage{epstopdf}
\usepackage{amsthm}
\DeclareMathOperator{\sign}{sign}

\newtheorem{theorem}{Theorem}
\newtheorem{corollary}[theorem]{Corollary}

\newtheorem{lemma}{Lemma}
\newtheorem{proposition}{Proposition}

\newcommand{\rmat}{{}_{-S}\mathcal{M}}

\newcommand{\reorientation}{{}_{c}\mathcal{M}_{\mathcal{A}}}
\newcommand{\rlawrence}{\mathcal{M}_{\mathcal{A}}}
\newcommand{\rmatrix}{{}_{c}\mathcal{A}}
\newcommand{\rmatrices}{{}_{S}\mathcal{A}}

\newcommand{\conv}{\mathit{conv}}
\newcommand{\relint}{\mathit{relint}}
\newcommand{\TT}{\mathrm{TT}}
\newcommand{\LT}{\mathrm{LT}}
\newcommand{\UD}{\mathrm{UD}}
\newcommand{\LD}{\mathrm{LD}}

\begin{document}

% A short title is not required, but if needed use:
% \title[short title]{full title}
\title{Projective Equivalences of k-neighbourly Polytopes}

\author{N. Garc\'{i}a Col\'{i}n$^{\dag}$}
\address{${\dag}$ Instituto de Matem\'{a}ticas, \'{A}rea de la Investigaci\'{o}n Cient\'{i}fica,
Circuito Exterior, Ciudad Universitaria, Coyoac\'{a}n, 04510. M\'{e}xico, D.\ F.}
\email{garciacolin.natalia@gmail.com}

% For each additional author, add another set of
% \author, \address, and \email commands
%
\author{D.G. Larman$^{*}$}
\address{${*}$ Department of Mathematics, University College London, Gower Street \\ London WC1E 6BT, UK}
\email{dgl@math.ucl.ac.uk}

\date{March 29th, 2011 and in revised form (date2).}
\subjclass[2000]{}

% \thanks entries are to acknowledge grants. You may combine
% all acknowledgments into one \thanks entry, or may use
% multiple \thanks entries. They generate footnotes without
% tags, so you must be explicit about which authors are
% thanking whom.
\thanks{}

\begin{abstract}
We prove the following theorem, which is related to McMullen's problem on projective transformations of polytopes; let $2\leq k\leq \lfloor{\frac{d}{2}}\rfloor$ and $\nu{(d, k)}$ be the largest number such that any set of $\nu{(d,k)}$ points lying in general position in $\mathbb{R}^d$ can be mapped by a permissible projective transformation onto the vertices of a k-neighborly polytope, then $d + \left\lceil{ \frac{d}{k}}\right\rceil +1  \leq \nu{(d, k)} < 2d-k +1$.
\end{abstract}

\maketitle

% Insert paper text here
\section{Introduction}

In 1970, Peter McMullen posed the following problem:

Determine the largest number $\nu{(d)}$ such that any set of $\nu{(d)}$ points lying in general position in ${\mathbb{R}}^d$ can be mapped by a permissible projective transformation onto the vertices of a convex polytope.

Considering the Gale diagram of the set of points in the problem above, D.\ Larman reformulated the question as follows:

Determine the smallest number $\lambda{(d)}$ such that for any set $X$ of $\lambda{(d)}$  points in $\mathbb{R}^d$ there exists a partition of $X$ into two sets $A$ and $B$ such that 
\[\conv{(A\backslash {x})}\cap \conv{(B\backslash {x})}  \neq   \emptyset, \quad \forall  \; x \epsilon X.\]

Where the relationship between $\nu$ and $\lambda$ is
\[\lambda(d)= \min_{w \in \mathbb{N}}\{ w | w \leq \nu(w-d-2) \}.\]

Using the reformulation, D.\ Larman \cite{DL} found the lower bound $ {2d +1} \leq \nu{(d)}$ by proving that $\lambda{(d)} \geq 2d+3$. He also proved that this bound is sharp in the cases where $d=1, 2$ and $3$, by constructing sets of four, six and eight points which do not have the required partition, as stated above. This supports his conjecture that the lower bound is sharp for higher dimensions.  He also found a set of $(d+1)^{2}$ points  in $\mathbb{R}^{d}$  such that no projective transformation maps them into the vertices of a convex polytope, thus proving
\[ {2d +1} \leq \nu{(d)} < {(d+1)^2}.\]

In 2001, using computational methods  D.\ Forge, M.\ Las Vergnas and P.\ Schuchert \cite{FORGE} found a divisible configuration of 10 points in dimension 4, confirming the conjecture for d=4.

I.\ Da Silva and R.\ Cordovil \cite{CORDOVIL} obtained a different equivalence of the original problem, using the fact that the matroid of a set of points in general position spanning $\mathbb{R}^d$ is a uniform oriented matroid of rank $d+1$:

Determine the smallest integer $n=\nu{(r)}$ such that for any orientation of the uniform rank $r=d+1$ oriented matroid  on a set of $\nu(r)$ elements, $\mathcal{M}$, there is an acyclic reorientation of $\rmat$ that has a positive circuit.

Using the reformulation above, in 1986, M.\ Las Vergnas \cite{LV} proved that
\[\nu{(d)} \leq \frac{(d+1)(d+2)} { 2}. \]

This bound was further refined by J.\ Ram\'{i}rez Alfons\'{i}n \cite{RA} to
\begin{equation} \label{upperbound}
\nu(d) < 2d  + \left\lceil{ \frac{d+1}{2}} \right\rceil.
\end{equation}

Building on the work by I. Da Silva, R. Cordovil and J. Ramirez-Alfonsin, we prove that the following two theorems, and their equivalence:

\begin{theorem}\label{PROJECTIVE}
Let $2\leq k\leq \lfloor{\frac{d}{2}}\rfloor$ and $\nu{(d, k)}$ be the largest number such that any set of $\nu{(d,k)}$ points lying in general position in $\mathbb{R}^d$ can be mapped by a permissible projective transformation onto the vertices of a k-neighborly polytope, then $d + \left\lceil{ \frac{d}{k}}\right\rceil +1  \leq \nu{(d, k)} < 2d-k +1$.
\end{theorem}

\begin{theorem}\label{TPARTITION} Let $\lambda{(d, k)}$ be the smallest number such that for any set $X$ of $\lambda{(d,k)}$  points in $\mathbb{R}^d$ there exists a subdivision of X into two sets $A, B$ such that 
\[\conv{(A\backslash \{x_1, x_2,\dots ,x_k\})}\cap conv{(B\backslash \{x_1, x_2,\dots , x_k\})}  \neq   \emptyset\] $ \forall \: \: 2\leq k\leq \lfloor\frac{d}{2}\rfloor$,  $ {\{x_1, x_2, \dots , x_k\}}  \subset X$. Then $2d + k +1 \leq \lambda{(d,k)} \leq (k+1)d + (k+2)$.
\end{theorem}

\section{Equivalence of Theorems 1 and 2}\label{equivalent}

Before we actually prove theorems 1 and 2, we begin by proving that the two statements are equivalent, as the actual proof of the bounds stated in the problems takes full advantage of their equivalence.

\begin{lemma} The following two problems are equivalent:
\begin{itemize}
\item[] Determine the largest number $\nu{(d, k)}$ such that any set of $\nu{(d,k)}$ points lying in general position in $\mathbb{R}^d$ can be mapped by a permissible projective transformation onto the vertices of a \emph{$k$-neighbourly} polytope. \\

\item[] Determine the smallest number $\mu{(d, k)}$ such that for any set $X$ of $\mu{(d,k)}$ points lying in general linear position on $S^{d-1}$, it is possible to choose a sequence $E=(\epsilon_{1},\ldots, \epsilon_{\mu(d,k)}) \in \{1,-1\}^{\mu(d,k)}$ such that for every $k$-membered subset of $X_{E}$, $ X_{E}^{k}$,  $0 \in conv( X_{E} \backslash X_{E}^{k})$, where  $X_{E}=\{\epsilon_{1}x_{1},\ldots, \epsilon_{\mu(d,k)}x_{\mu(d,k)}\}.$ \\
\end{itemize}

and the relationship between $\nu({d,k})$ and $\mu({d,k})$ is
\begin{align*} 
\nu{(d,k)} &= \max_{w \in \mathbb{N}} { \{ w \geq \mu{(w-d-1,k)} \} } ,\\
\mu{(d,k)} &= \min_{w \in \mathbb{N}} { \{ w \leq \nu{(w-d-1,k)} \} } .
\end{align*}
\end{lemma}

\begin{proof} Let $X$ be a set of points in general position  in $\mathbb{R}^{d}$ such that $|X|=\nu \leq \nu(d,k).$  By hypothesis and the properties of Gale transforms, there is a nonsingular projective transformation, permissible for $X$, $P(x)=\frac{Ax+b}{\langle c,x \rangle + \delta}$, such that P(X) is the set of vertices of a $k$-neighbourly convex polytope. Then the Gale diagram of $X$, $\overline{X}$, is linearly equivalent to the set $\overline{X}_{E}=\{\epsilon_{1}\overline{x}_{1},\ldots, \epsilon_{\nu}\overline{x}_{\nu}\}$, where $\epsilon_{i}=\sign( \langle c,x_{i} \rangle + \delta)$ for all $i=1,\ldots \nu.$ It has been proven in \cite{Grunbaum} that for all $k$-membered subsets of $\overline{X}_{E}$, $ \overline{X}_{E}^{k}$,  $0 \in conv( \overline{X}_{E} \backslash \overline{X}_{E}^{k}).$ So $\nu \geq \mu(\nu-d-1,k).$

Conversely, let $\overline{X} \in \mathbb{R}^{d}$ such that $|\overline{X}|=\mu \geq \mu(d,k)$, is the Gale diagram of a set $X\subset \mathbb{R}^{\mu-d-1}.$ Then there is a sequence  $E=(\epsilon_{1},\ldots, \epsilon_{\mu}) \in \{1,-1\}^{\mu}$ such that $\overline{X_{E}}=\{\epsilon_{1}\overline{x_{1}},\ldots, \epsilon_{\mu}\overline{x_{\mu}}\}$ is the Gale diagram of a $k$-neighbourly polytope, where $\epsilon_{i} \in \{1, -1\}.$ Using the properties of the Gale transform \cite{Grunbaum}, there  are $c \in \mathbb{R}^{d}$ and $\delta \in \mathbb{R}$ such that $\epsilon_{i}=\langle c, x_{i}\rangle + \delta$ for all $i=1,\ldots,\mu$, and a linear transformation $A$ and a vector $b \in \mathbb{R}^{d}$ such that the projective transformation $P(x)=\frac{Ax+b}{\langle c,x \rangle + \delta}$ is regular and permissible for $X$, and such that $P(X)=X_{E}$, where $X_{E}$ is the Gale transform of $\overline{X_{E}}.$ Hence $\mu \leq \ \nu(\mu-d-1,k).$
\end{proof}

\begin{lemma} The following two problems are equivalent:
\begin{itemize}
\item[] Determine the smallest number $\lambda{(d, k)}$ such that for any set $X$ of $\lambda{(d,k)}$  points in $\mathbb{R}^d$ there exists a subdivision of X into two sets $A, B$ such that $ \conv{(A \backslash {\{x_1, x_2, \dots , x_k\}})} \cap \conv {(B \backslash {\{x_1, x_2, \dots , x_k\}})}  \neq \emptyset$, for all $ {\{x_1, x_2, \dots , x_k\}}  \subset X.$\\

\item[] Determine the smallest number $\mu{(d, k)}$ such that for any set $X$ of $\mu{(d,k)}$ points lying in linearly general position on $S^{d-1}$, it is possible to choose a sequence $E=(\epsilon_{1},\ldots, \epsilon_{\mu(d,k)}) \in \{1,-1\}^{\mu(d,k)}$ such that for every $k$-membered subset of $X_{E}$, $ X_{E}^{k}$,  $0 \in \conv ( X_{E} \backslash X_{E}^{k})$, where  $X_{E}=\{\epsilon_{1}x_{1},\ldots, \epsilon_{\mu(d,k)}x_{\mu(d,k)}\}.$ \\
\end{itemize}

So $\mu({d+1,k})=\lambda({d,k}).$
\end{lemma}

\begin{proof} Let $X$ be a set of $\mu < \mu(d+1,k)$ points lying in general linear position in ${S}^d$. Then for all sequences $E=(\epsilon_{1},\ldots, \epsilon_{\mu}) \in \{1,-1\}^{\mu}$, there is a $k$-membered set,  $X^{k}=\{{x}_{i_{1}}\ldots {x}_{i_{k}}\} \subset {X}$, such that $0 \not \in  \relint \conv ( X_{E} \backslash X_{E}^{k})$, where $X_{E}=\{\epsilon_{1}x_{1},\ldots, \epsilon_{\mu}x_{\mu}\}$ and $X_{E}^{k}=\{\epsilon_{i_{1}}x_{i_{1}},\ldots, \epsilon_{i_{k}}x_{i_{k}}\}$.

Therefore there is a hyperplane $H'$ that weakly separates the origin from $\conv ({X}_{E} \backslash X_{E}^{k}).$ However, as the points in ${X}$ are in general  linear position, there is a hyperplane $H$ through the origin such that \[{X}_{E} \backslash X_{E}^{k} \subset S^{\nu-d-2}\cap H^{+}.\]

Then, given any $E \in \{1,-1\}^{\mu}$, consider the partition of ${X}$ formed by the sets
\begin{gather*}A =\{{x}_{i}| \epsilon_{i} \in E \text{ is such that } \epsilon_{i}=+\}\\
 B =\{{x}_{i}| \epsilon_{i} \in E \text{ is such that } \epsilon_{i}=-\}.\end{gather*} 
 For each $E$, the set $X^{k}$ induces a hyperplane $H$ as above such that $H$ separates $\conv (A \backslash X^{k})$ from $\conv (B \backslash X^{k}).$ This implies that \[\lambda(d, k) \geq \mu(d+1, k).\]

Conversely,  if a set of points $X=\{x_{1},\ldots,x_{\lambda}\}$ lies in an open hemisphere of  $S^{d}$ and is not $k$-divisible, then there exists $\eta >0$ such that every set $X'=\{x_{1}',\ldots,x_{\lambda}'\}$ with $\lVert x_{i}-x_{i}'\rVert<\eta$ is not $k$-divisible and lies in the same hemisphere. Consequently, it can be supposed that $X$ is in general linear position.

Given any sequence $E =(\epsilon_{1},\ldots, \epsilon_{\lambda}) \in \{1,-1\}^{\lambda}$, with $\lambda < \lambda(d,k)$, consider the partition into two sets given by
\begin{gather*}{A} =\{{x}_{i}| \epsilon_{i} \in E \:\text{is such that} \:\epsilon_{i}=+\}\\
{B} =\{{x}_{i}| \epsilon_{i} \in E \:\text{is such that} \:\epsilon_{i}=-\}.\end{gather*}
By hypothesis there are points $X^{k}=\{x_{i_{1}},\ldots,x_{i_{k}}\} \in X$ such that
\[\conv ( A \backslash \{x_{i_{1}},\ldots,x_{i_{k}}\} ) \cap \conv ( B \backslash \{x_{i_{1}},\ldots,x_{i_{k}}\} ) = \emptyset.\] Thus there is a hyperplane $H$ through the origin that separates
\[
\conv ( A \backslash \{x_{i_{1}},\ldots,x_{i_{k}}\} ) \text{ from } \conv ( B \backslash \{x_{i_{1}},\ldots,x_{i_{k}}\} ).
\]

Hence ${A}_{E}\backslash \{\epsilon_{i_{1}}{x}_{1},\ldots,\epsilon_{i_{k}} {x}_{i_{k}}\}$ and ${B}_{E}\backslash \{ \epsilon_{i_{1}} {x}_{1},\ldots,\epsilon_{i_{k}} x_{i_{k}} \}$ are contained in the same open half space, with ${A}_{E}=\{\epsilon_{i} {x}_{i}|{x}_{i} \in {A} \}$ and $B_{E}=\{ \epsilon_{i} {x}_{i}|{x}_{i} \in B \}$, which proves that for all $E \in \{1,-1\}^{\lambda}$, there is a set $\{{x}_{i_{1}},\ldots {x}_{i_{k}}\} \subset {X}$ such that $0 \not \in \conv ({X}_{E} \backslash \{ {x}_{i_{1}},\ldots,{x}_{i_{k}}\}).$ Then  $\mu(d+1, k) \geq \lambda(d,k).$\end{proof}

From the two lemmas above, we have the final relationship between $\lambda$ and $\nu$;

\begin{corollary}
\begin{align*} \nu{(d,k)} &= \max_{w \in \mathbb{N}} { \{ w \geq \lambda{(w-d-2,k)} \} } ,\\
\lambda{(d,k)} &= \min_{w \in \mathbb{N}} { \{ w \leq \nu{(w-d-2,k)} \} } .\end{align*}
\end{corollary}

%%%%%%%%%%%%%%%%%%%%%%%%%%%%%%%%%%%%%%%%%%%%%
\section{Proofs of theorems 1 and 2}

The proofs of the upper and lower bounds of Theorems 1 and \ref{TPARTITION} use several results on realizable oriented matroids. In the following subsection we briefly outline these results and refer the reader to \cite{OM}  and \cite{RA} for proofs and further background. We also introduce a few concepts that will be extensively used in the proof of the upper bound.

\subsection{Oriented Matroids}

The class of Lawrence oriented matroids of rank $r$ on a ground set of cardinality $n$, as defined by J.Lawrence \cite{OM} is comprised by unions of rank 1 oriented matroids on a totally ordered set $(E, <)$, $\mathcal{M}=\bigcup_{i=1}^{r} \mathcal{M}_{i}$. Each oriented matroid in the class can be represented by a matrix $ \mathcal{A} = (a_{i,j})$, with $ 1\leq i \leq r$ and  $ 1 \leq j \leq n$, whose entries are in the set ${\{1,-1}\}$. 

Also,  if  $\mathcal{A} = (a_{i,j})$ is as before, we denote $\rlawrence$ its corresponding oriented matroid. Each element of the ground set of $\rlawrence$ will be associated to a column of $\mathcal{A}$. Given $c \in E$ where $E$ is the ground set of $\rlawrence$, the matrix corresponding to the matroid reorientation over $c$, denoted $\reorientation$,  is obtained by multiplying by $-1$ the sign of all the entries in the column corresponding to the element $c$ of $\mathcal{A}$, denoted $\rmatrix$.

By definition, an interior element of an oriented matroid $\mathcal{M}$ is an element $c$ of the ground set of the matroid such that there is a circuit $X$ with $X^{+}=\{c\}.$  For an interior element $c$, the reorientation  $\reorientation$, $of \rlawrence$ is cyclic.

It is known that the matrix $\mathcal{A}$ encodes the chirotope of the matroid $\rlawrence$ in the following way:
\begin{displaymath}
\mathcal{X}(B)=  \prod_{i=1}^r a_{i, j_i},
\end{displaymath}
where $E=\{e_{{1}}, \ldots, e_{n}\}$ is the ground set and $B=\{e_{j_1} < \cdots < e_{j_r}\} \subset E$ is a basis of the matroid.

Thus, the signature of every circuit can be read from the chirotope as
\[\mathcal{X}(B)=\mathcal{X}(e_{j_{1}}, \ldots, e_{j_{r}})=\mathcal{X}_{j_{1}}(e_{j_{1}})\cdots \mathcal{X}_{j_{r}}(e_{j_{r}}),\]
where $\mathcal{X}$ is the chirotope of $\mathcal{M}= \cup_{i=1}^{r} \mathcal{M}_{i}$ and $\mathcal{X}_{i}$ is the chirotope of $\mathcal{M}_{i}.$

It is also known that the sign of the element $x_i$ in a circuit $C$ is  \[C(e_{j_{i}})= (-1)^i \cdot \mathcal{X}(e_{j_{1}},\ldots,e_{j_{i-1}},e_{j_{i+1}},\ldots,e_{j_{r}}),\] so that
\[C(e_{j_{i}})= (-1)^i \cdot \mathcal{X}_{j_{1}}(e_{j_{1}})\cdots \mathcal{X}_{j_{i-1}}(e_{j_{i-1}})\cdot\mathcal{X}_{j_{i+1}}(e_{j_{i+1}})\cdots \mathcal{X}_{j_{r}}(e_{j_{r}}).\]
In the matrix representation this means that if $C=\{e_{j_{1}}, \ldots, e_{j_{r}}\}$ is a circuit with $j_{i} \in \{1, \ldots, n\}$, then
\[C(e_{j_{i}})= (-1)^i \cdot a_{1,j_{1}}\cdots a_{i-1,j_{i-1}}\cdot a_{{i}, j_{i+1}}\cdots a_{{r-1}, j_{r}}.\]

Hence $C(e_{j_{i}})\cdot C(e_{j_{i+1}})= - a_{i,j_{i+1}} \cdot a_{i, j_{i}}$. So that $C(e_{j_{i}})= C(e_{j_{i+1}})$ if and only if $a_{i,j_{i+1}} =- a_{i, j_{i}}.$ 

Using all these properties of Lawrence oriented matroids , Ram\'{i}rez-Alfons\'{i}n introduced the following definitions for $\mathcal{A}$:

A \textbf{Plain Travel} (PT) in $\mathcal{A}$ is the following  subset of the entries of \(\mathcal{A}$:
\[PT=\{ [a_{1, 1}, a_{1, 2},\dots ,a_{1, j_{1}}] , [a_{2, j_{1}}, a_{2, j_{1}+1},\dots , a_{2, j_{2}}] , \dots , [a_{s, j_{s-1}}, a_{s, j_{s-1}+1},\dots , a_{s, j_{s}}] \} \]  with \( 2\leq j_{i-1} \leq j_{i}  \leq n \quad \forall \quad 1 \leq i \leq r ,\;\; 1\leq s \leq r$   and  $j_{s} = n$.

A \textbf{Top Travel} (TT) in $\mathcal{A}$ is a $PT$ with the following additional constraints:
\begin{enumerate}
\item$ a_{i, j_{i-1}} \times a_{i, j}= 1, \quad \forall \quad\  j_{i-1} \leq j \leq j_{i};$
\item$ a_{i, j_{i-1}} \times a_{i, j_{i}}= -1;$  and
\item either 
\begin{itemize}
\item $1\leq s < r$;   then  $j_{s} = n$ or 
\item $s=r$ and $ j_{s} \leq n.$
\end{itemize}
\end{enumerate}

A \textbf{Bottom Travel} (BT) in $\mathcal{A}$ is defined as a TT starting from the bottom left corner of the matrix; i.e.,
\begin{enumerate}
\item $ a_{i, j_{i+1}} \times a_{i, j}= 1, \quad \forall \quad\  j_{i} \leq j \leq j_{i+1};$
\item $ a_{i, j_{i+1}} \times a_{i, j_{i}}= -1;$  and 
\item either 
\begin{itemize}
\item $1 < s \leq r$;   then  $j_{s} = 1$ or 
\item $s=1$ and $ 1 \leq j_{s}.$
\end{itemize}
\end{enumerate}

Plain travels  can then be associated with circuits of the matroid thus, in order  to study cyclicity in the matroid, one only needs to study the behaviour of travels in $\mathcal{A}$. In ~\cite{RA}, J. Ramirez-Alfonsin proves the following propositions:

\begin{proposition} \label{LLOM} Let  $ \mathcal{A} = (a_{i,j})$ with $1\leq i \leq r$, $1 \leq j \leq n$, be a matrix with entries from  ${\{1,-1}\}$, $\mathcal{M}_{\mathcal{A}}$ its corresponding Lawrence oriented matroid, and TT and BT the top and bottom travels constructed on $\mathcal{A}$. Then the following conditions are equivalent:
\begin{enumerate}
\item$\mathcal{M}_{\mathcal{A}}$ is \emph{cyclic};
\item TT ends at $ a_{r,s}$ for some $ 1 \leq s < n$; and
\item BT ends at $ a_{1, s'}$ for some $ 1< s' \leq n$.
\end{enumerate}
\end{proposition}

\begin{proposition} \label{PT} Let  $ \mathcal{A} = (a_{i,j})$ with $  1\leq i \leq r$,  $ 1 \leq j \leq n$, be a matrix with entries from  ${\{1,-1}\}$ and  $\mathcal{M}_{\mathcal{A}}$ its corresponding Lawrence oriented matroid. Then there is a bijection between the set of all plain travels of $\mathcal{A} $ and the set of all acyclic reorientations of $\mathcal{M}_{\mathcal{A}}$.
\end{proposition}

In order to show the bound in equation  \eqref{upperbound}, stated in the introduction; Ram\'{i}rez-Alfons\'{i}n \cite{RA} constructs a family of Lawrence oriented matroids $\mathcal{A}$ of rank $r$ on a ground set $E$ with cardinality $n$. For this family of matroids it is always sufficient to reorient one of the elements in order to make them cyclic. That is, after just one column reorientation in the matrix $\mathcal{A}$, which represents the matroid, either TT ends at the last row or BT ends at the first row.
 
Ram\'{i}rez-Alfons\'{i}n's  construction of the desired families of matrices consists of restricting the patterns of signs formed by the matrices' elements. A visual technique for defining how patterns are formed in the grid of signs is constructing the chessboard of the matrix.

The \textbf{chessboard} of the matrix  $\mathcal{A}$ is a black and white board of size $(r-1)*(n-1)$, in which the square $s(i,j)$ has its upper left hand corner at the intersection of row $i$ and column $j$. A square $s(i,j)$, with $1 \leq i \leq r-1$ and $ 1 \leq j \leq n-1$, will be \emph{said to be} black if the product of the entries $a_{i,j}, a_{i,j+1}, a_{i+1,j}, a_{i+1,j+1}$ is $-1$, and white otherwise. 

It is easy to check that chessboards are invariant under reorientations of $\mathcal{A}$, hence they provide a natural framework for studying the information encoded in $\mathcal{A}$, because despite the many combinations of patterns of signs possible in a matrix, the analysis can be reduced to types of chessboards.

Chessboards have the following property: if there is one black square between a top travel TT and  a bottom travel BT, they follow symmetrically opposite paths through the entries of the matrix. In other words, if TT makes a single horizontal movement from $a_{i,j}$ to $a_{i,j+1}$ and continues its movement forward in the same row, then BT goes from  $a_{i+h,j+1}$ to $a_{i+h,j}$ and moves vertically to  $a_{i+h-1, j}$ (with $h\geq{1}$), and the other way around.

For the proof that follows we will use subfamilies of Lawrence oriented matroids whose matrix representation have a chessboard whose only black tiles are along its diagonal.  In such a chessboard we can define UD and LD as the following sets of elements of $\mathcal{A}:$
\begin{align*}
UD&=\left\{ a_{i,j}\;|\; s(i,j)\text{ or }s(i,j-1)\text{ is black}\right\}\\
LD&=\left\{ a_{i,j}\;|\; s(i-1,j-1)\text{ or }s(i-1,j)\text{ is black}\right\}.
\end{align*}

That is, UD consists of all the elements touching the diagonal of black blocks from above, and LD are the elements touching the diagonal from below.

Since we use matroids to prove Theorem \ref{TPARTITION}, a translation of geometrical neighbourliness into matroid neighbourliness is needed.

Recall that a d-polytope is $k$-neighbourly if given $k\leq{\lfloor{\frac{d}{2}}\rfloor}$ fixed , every subset of at most $k$ vertices of the vertex set of the polytope is contained in the vertex set of a facet of the polytope. Also recall that a subset $F$ of the ground set of a uniform matroid polytope $\mathcal{M}$ is a face of the matroid if and only if for all circuits $C$ of $\mathcal{M}$, $C^{+} \not \subset F$.

These necessarily imply that  a matroid polytope is \textbf{\boldmath{$k$}-neighbourly} iff $k \leq |C^+|$ and  $k \leq |C^-|$, $\forall$ $C \in \mathcal{C}$, where $\mathcal{C}$ is the set of circuits of the matroid $\mathcal{M}$, $C^{+}$ and $C^{-}$ are the positive and the negative elements of the circuit $C$, respectively.

In particular, matroid polytopes are acyclic; thus a matroid polytope has  \emph{an acyclic reorientation of k or fewer ``interior'' points} iff there is at least one $C \in \mathcal{C}$ such that $ |C^+| \leq k$ or $ |C^-| \leq k$.

%%%%%%%%%%%%%%%%%%%%%%%%%%%%%%%%%%%%

\subsection{Proof of the Upper Bound}

In order to find an upper bound for Theorem \ref{PROJECTIVE}, it is therefore sufficient to find families of realisable matroids such that any acyclic reorientation of them contains at least one $C \in \mathcal{C}$ such that $ |C^+| \leq k$ (or $ |C^-| \leq k$). 

As before, considering the matrix representation $\mathcal{A}=(a_{i,j})$ of each element in the class of Lawrence oriented matroids of unions of $r$ uniform rank one oriented matroids over a ground set $E$, with cardinality $n$ and using Propositions \ref{LLOM} and \ref{PT}, we only need to find a family of acyclic matrices such that there is always a set $S$ of indices of columns of the matrix with $|S|\leq k$ such that the reorientation $\rmatrices$ is cyclic. This will be achieved by considering a class of acyclic matrices with a specific chessboard, and proving that a suitable set $S$ can always be found. The rest follows as all Lawrence oriented matroids are geometrically realizable.

Let $k=2$, we define the set of matrices with the same chessboard, $CB(r,n, 2)= \{ \mathcal{A} = (a_{i,j})|\: 1\leq i \leq r, \: 1 \leq j \leq n=2(r-1)+ 1,\: \text{and} \: a_{i,j} \in \{1,-1\} \}$, such that:

\renewcommand{\labelitemii}{$\circ$}
\begin{itemize}
\item$s(i,j)=a_{i,j}\times a_{i,j+1}\times a_{i+1,j}\times a_{i+1,j+1}=-1$ (i.e. $s(i,j)$ is black) if 
\begin{itemize}
\item$j=2(i-1)+1 $, or 
\item
$j+1=2(i)+ 1$;
\end{itemize}
\item$s(i,j)=1$ (i.e. $s(i,j)$ is white) otherwise.

\end{itemize}

This chessboard consists of  black steps of length two along the diagonal, as seen in Figure~\ref{fig: 5x9}.

%tiene que entrar la imagen del chesboard
\begin{figure}[t]
\begin{center}
\includegraphics[width=7cm]{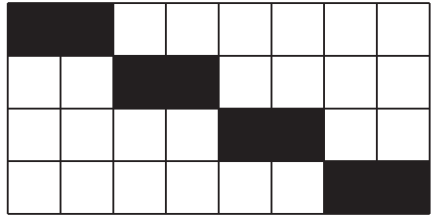}
\caption{Chessboard for a $5\times{9}$ matrix}
\label{fig: 5x9}
\end{center}
\end{figure}

\begin{lemma} \label{LBASE} A matrix  $\mathcal{A} \in CB(r,n, 2)$ has a cyclic reorientation $\rmatrices$, where $|S|\leq 2$.
\end{lemma}

\begin{proof} The proof will follow by induction on r. The first interesting case is when $r=3$; then $n=5.$ There are five different cases where $\mathcal{A}$ is acyclic, which are characterized by their BT and TT, shown in Figure~\ref{fig:fivecases}. In these five cases, $\mathcal{A}$ always has a cyclic reorientation, where the reoriented set has cardinality less or equal to $2$. Working from the top left hand corner in clockwise order in the figure, the columns that can be reoriented to make the chessboards cyclic are given in the following table:

\begin{center}
\begin{tabular}[pos]{cc}
chessboard & columns to be reoriented\\
1st & 4th\\
2nd & 3rd\\
3rd & 2nd\\
4th & 1st and 4th\\
5th & 2nd.\\
\end{tabular}
\end{center}

Suppose that for all $r < r^{*}$, the $r \times 2(r-1)+1$ matrix $\mathcal{A}$  has a cyclic reorientation of fewer than 2 elements.

Let $r=r^{*}$ and assume that TT last intersects $UD\cap LD$ in $a_{i,j}$ with $j=2(i-1)+1$ and $i<r.$ If $2 \leq i$, the lemma holds by the induction hypothesis. Equally, suppose BT last intersects $UD \cap LD$ (from right to left) at an element $a_{i',j'}$ with $j'=2(i'-1)+1$. If $ i' \leq r-1$, again the lemma holds.

Then TT has to go through elements $\{a_{1,1}, a_{1,2}, a_{1,3}, a_{1,4}\}$ and it always travels above UD, BT always travels below LD, and $\mathcal{A}$ is acyclic. TT finishes at an element $a_{i', n}$. These observations imply that if TT makes $2(r-1)+1-3$ horizontal movements and $i'-1$ vertical movements in order to reach column $2(r-1)+1$ from column $3$, as BT always passes strictly below UD, it always has the opposite behaviour to TT. Hence, BT has to make precisely $2r-i'-3$ vertical movements before column $3$. But $i' < r$, so $2r-i'-3\geq {r-3}$. That is, at column $3$, BT is already in row $2$, and by reorienting just one column the result follows. \end{proof}

 %dibujito%
\begin{figure}[t]
\begin{center}
\includegraphics[width=12cm]{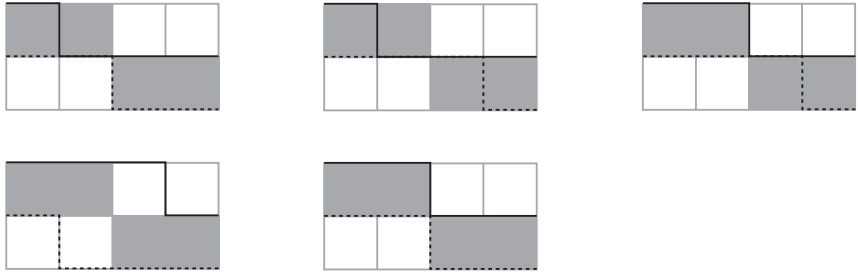}
\caption{Five cases where $\mathcal{A}$ is acyclic}
\label{fig:fivecases}
\end{center}
\end{figure}

%Aqui acaba la prueba del lemma 

When $k \geq 3$, the chessboard which is suitable for proving the lemma equivalent to Lemma \ref{LBASE} is constructed in the following manner.

 %Definicion del chessboard generalizado
 
Let $CB(r,n, k)= \{ \mathcal{A} = (a_{i,j})|\: 1\leq i \leq r, \: 1 \leq j \leq n=2(r-1)-(k-2) + 1,\: \text{and} \: a_{i,j} \in \{1,-1\} \}$ such that; 
\begin{itemize}

\item$s(i,j)=a_{i,j}\times a_{i,j+1}\times a_{i+1,j}\times a_{i+1,j+1}=-1$ (i.e. $s(i,j)$ is black) if 

\begin{itemize}

\item $j=2(i)- \lceil{\frac{(i-1)+l}{s}}\rceil$, or 
\item $j=2(i)- \lceil{\frac{(i-1)+l}{s}}\rceil+ 1$ and $i+s-l \not\equiv{{0}\mod{s}};$ 
\end{itemize}

\item$s(i,j)=1$ (i.e. $s(i,j)$ is white) otherwise,

\end{itemize}
where $s= \lceil {\frac{r-1}{k-2}} \rceil$, $3 \leq k < \frac{r-1}{2}$ and $ 1 \leq l \leq s$ are fixed.

This chessboard has black diagonals made up of single black blocks evenly distributed among double blocks. Figure~\ref{fig: varioscasos} illustrates this chessboard for the cases where $r=8$, $k=3$ and $l=1,4.$

\begin{lemma}\label{GENERAL} A matrix  $\mathcal{A} \in CB(r,n,k)$ has a cyclic reorientation $\rmatrices$ with $|S| \leq k$ for all $3\leq k \leq \lfloor \frac{r}{2} \rfloor.$
\end{lemma}

\begin{proof} 

As before, this proof will also work by induction for both $k$ and $r.$ Let $k=3$. Although  the matrix  $\mathcal{A}$ represents a matroid and therefore $r > 7$, the purely combinatorial property holds for chessboards with $3\leq r.$  The proof will follow by induction on $r$, so first consider the case $r=3.$ In this case the chessboard has four columns and three rows, and it is easily seen that for any TT of an acyclic matrix, three reorientations are more than enough to make the travel end at row $3$.

Now suppose the lemma holds for all $r < r^{*}.$ Let $r=r^{*}$, so $n=2r-2$ and $s=r-1.$ Then there is precisely one single black block in the diagonal.  Both $\TT \cap \UD \cap \LD \neq \emptyset$ and $\LT \cap \UD \cap \LD \neq \emptyset.$ Let $i$ be the largest $1\leq i \leq r$ such that $a_{i, j}\in \TT \cap \UD \cap \LD$ or the smallest $i$ such that $a_{i, j}\in \LT \cap \UD \cap \LD$.

If $1< i$ for TT  or $i< r$ for LT, by the induction hypothesis, the lemma holds. 

Suppose $i=1.$ If $l=1$, then TT takes the elements $\{  a_{1,1}, a_{1,2}, a_{1, 3}\}.$ Hence if column one is reoriented, the new top travel, $\TT'$, takes the elements $\{  a_{1,1}, a_{1, 2}, a_{2, 2}, a_{2, 3}\}.$ But $a_{2, 2} \in \UD \cap \LD$ and after column two there are only double blocks in the diagonal so, by Lemma \ref{LBASE}, the lemma holds. 

If $l>1$,  then TT takes elements $\{  a_{1,1}, a_{1, 2}, a_{1, 3} a_{1, 4}\}$ and reorienting column one,  the new top travel, $\TT'$, takes elements $\{  a_{1,1}, a_{1, 2}, a_{2, 2}, a_{2, 3}, a_{3,3},a_{3,4}\}$, hence traveling below UD.

If $\TT'$ never crosses UD again the lemma holds. 

Therefore, suppose  $\TT'\cap \UD\cap \LD \neq \emptyset.$ Let  $i$ be the smallest $1\leq i \leq r$ such that $a_{i, j}\in \TT' \cap \UD \cap \LD.$ By lemma \ref{LBASE}, if $i>l$ the lemma holds. Thus, the only case left is  $3\leq i \leq l$ and $j=2i-1.$

%%%%%%%%%%%%%%%%%%FIGURA%%%%%%%%%%%%%%%%%%%%%%%%%%

\begin{figure}[t] %  figure placement: here, top, bottom, or page
\centering
\includegraphics[width=12cm]{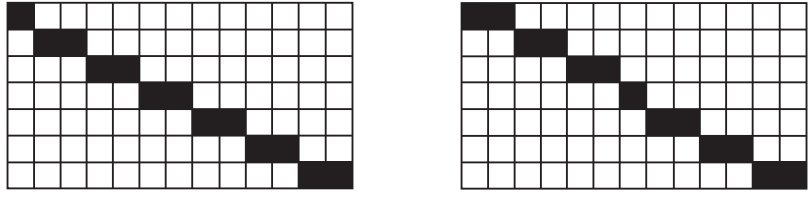} 
\caption{Two different types of chessboards valid for rank $8$ matrices.}
\label{fig: varioscasos}
\end{figure}
%%%%%%%%%%%%%%%%%%%%FIGURA%%%%%%%%%%%%%%%%%%%%%%%%%

The original TT passes through an element $a_{i',2i-1}$ with with $i'<i.$ So between column $3$ and column $2i-1$, TT makes $2i-4$ horizontal movements and $i'-1$ vertical movements. Given that $\TT'$ and TT are strictly separated by the diagonal of black blocks, the number of vertical movements made by $\TT'$ between columns $3$ and $2i-1$ equals $2i-i'-3.$ But by hypothesis, the number of vertical movements $\TT'$ makes between columns $3$ and $2i-1$ is precisely $i-3.$ Then $i-3=2i-i'-3$, so $i=i'$, a contradiction.  

Then if  $a_{i, j}\in \TT' \cap \UD \cap \LD$, necessarily $i > l$ and, by Lemma \ref{LBASE}, the lemma holds for $k=3.$

Suppose now that for each $k< k^{*}$, the lemma holds for all $2k \leq r$. Let $k=k^{*}.$ Both $\TT \cap \UD \cap \LD \neq \emptyset$ and $\LT \cap \UD \cap \LD \neq \emptyset.$ Let $i$ be the largest $1\leq i \leq r$ such that $a_{i, j}\in \TT \cap \UD \cap \LD$ or the smallest $i$ such that $a_{i, j}\in \LT \cap \UD \cap \LD$. 

If $l< i$ for TT  or $i< (k-1)s +l$ for LT, by the induction hypothesis for $k$, the lemma holds. 

Suppose $i=l$. Then TT takes elements $\{  a_{i,j}, a_{i,j+1}, a_{i, j+2}\}.$ Hence if column one is reoriented,  the new top travel, $\TT'$, takes the elements $\{  a_{i,j}, a_{i, j+1}, a_{i+1, j+1}, a_{i+1, j+2}\}.$ But $a_{i+1, j+1} \in \UD \cap \LD$ and after column two there are only $k-3$ single blocks in the diagonal, so by the induction hypothesis, the lemma holds. 

If $i<l$  then TT takes the elements $\{  a_{i,j}, a_{i, j+1}, a_{i, j+2},  a_{i, j+3}\}$ and reorienting column j,  the new top travel, $\TT'$, takes elements $\{  a_{i,j}, a_{i, j+1}, a_{i+1, j+1}, a_{i+1, j+2}, a_{i+2,j+2},a_{i+2,j+3}\}.$ If $\TT'$ never crosses UD again the lemma holds. 

Therefore suppose  $\TT'\cap \UD\cap \LD \neq \emptyset.$ Let  $i'$ be the smallest $1\leq i' \leq r$ such that $a_{i', j'}\in \TT' \cap \UD \cap \LD.$ By the induction hypothesis, if $i'>l$, the lemma holds. Thus, the only case left is  $3\leq i' \leq l$ and $j'=2i'-1.$

The original TT passes through an element $a_{i'',2i'-1}$ with $i''<i'$. So between column $j+2$ and column $2i'-1$, TT makes $2i'-j-3$ horizontal movements and $i''-1$ vertical movements. Given that $\TT'$ and TT are strictly separated by the diagonal of black blocks between columns $j+2$ and $2i'-1$, the number of vertical movements made by $\TT'$  between those columns equals $2i'-j-i''-2.$ But by hypothesis, the number of vertical movements $\TT'$ makes between columns $j+2$ and $2i'-1$ is precisely $i'-j-2.$ Then $i'-j-2=2i'-j-i''-2$, so $i'=i''$, a contradiction.  

Therefore the lemma holds for all $3\leq k$ and $2k \leq r$.
\end{proof}

Summarizing, these two lemmas prove \[\nu{(d,k)} \leq 2d-k+1 \quad \forall k \geq 2,\] the upper bound.

\subsection{Proof of the Lower Bound}

 For the proof of the lower bound it is better to use the statement of the problem set in terms of partitions of points. 
 
\begin{lemma}\label{RADONSOTE} Let $X$ be a set of $(k+1)d+(k+2)$ points in general position in $\mathbb{R}^d$. Then there is a partition of $X$ into two sets,  $A, B$, such that $A \cap B = \emptyset$ and $A \cup B = X$, with the following property:
 \begin{gather*}
{\conv  \left(A \backslash \{x_1, x_2, \dots , x_k\} \right)} 
\bigcap 
{\conv  \left(B \backslash \{x_1, x_2, \dots , x_k\} \right)} 
\neq \emptyset, \\
 \forall \: \:    {\{x_1, x_2,\dots , x_k\}}  \subset X.\end{gather*}
\end{lemma}

\begin{proof} The proof will follow by induction. The case where  $k=1$ was proved by D.\ Larman in \cite{DL}. Let $k\geq 2$ and suppose the statement of the lemma is true for all $k < k^{*}.$ Then it only has to be proved that the lemma is true for $k=k^{*}$. 

For brevity, if a set of points X has the property stated in Lemma \ref{RADONSOTE}, then we call $X$  \emph{$k$-divisible}.

The $(k+1)d+(k+2)$ vertices of a cyclic polytope in $\mathbb{R}^{k(d+1)}$ are in general position. Through a Gale transform, these points can be transformed into a set $X$ of $(k+1)d+(k+2)$ points in $\mathbb{R}^{d}$ such that they are $k$-divisible. 

So there are $(k+1)d+(k+2)$ points in general position which are $k$-divisible. The property of being $k$-divisible is closed among all sets of $(k+1)d+(k+2)$ points in general position in $\mathbb{R}^d$. Let $ \{x_1, x_2, \dots , x_n \}$ be a $k$-divisible set. It is therefore enough to prove that if $ \{ y, x_1, x_2, \dots , x_n \}$, where $n= (k+1)d+(k+2)$, is a set of points in general position  in $\mathbb{R}^d$ then the set $ \{y, x_2, \dots , x_n \}$ is also $k$-divisible. 

Let T be the set of real numbers t such that \[X(t)=\left\{(1-t)x_1+ty, x_2, \dots , x_n | \: 0 \leq t \leq 1 \right\}\] is $k$-divisible. T is a non-empty closed subset of $[0,1]$. Suppose \[t_0= \sup_{t \in T}{t}< 1\] and let $x_1(t)=(1-t)x_1+ty$ for all $t \in \mathbb{R}.$ Then the set $X(t_0)$ is $k$-divisible with a subdivision $A(t_0)=\{x_1(t_0), x_2, \dots , x_r \}$ and $B(t_0)=\{x_{r+1}, x_2, \dots , x_n \}$ (with some relabeling possibly needed).

By definition, for each $t > t_0 $ there exist points \[\{x_{j_1}(t), x_{j_2}(t) , \dots  ,x_{j_k}(t)\}  \subset X(t)\] such that if $A(t)=\{x_1(t), x_2, \dots , x_r \}$ and $B(t)=\{x_{r+1}, x_2, \dots , x_n \}$, then 
\[ 
\conv  \left(A(t) \backslash \{x_{j_1}(t) , \dots  ,x_{j_k}(t)\} \right)
\cap 
\conv  \left(B(t) \backslash \{x_{j_1}(t), \dots  ,x_{j_k}(t)\} \right)
= \emptyset.\]

Since there are only finitely many combinations of $n$ points in subsets of size $k$, there is a sequence $t_n \to {t_0}^+$ as $n \to \infty$ such that $\{x_{j_1}(t_n), x_{j_2}(t_{n}) , \dots  ,x_{j_k}(t_n)\}$ is fixed and equal to $\{x_{j_1}, x_{j_2}, \dots  ,x_{j_k}\}.$ Also, for each $t > t_0$ there is a hyperplane $H(t)$ such that \[{\conv  \left(A(t) \backslash \{x_{j_1}(t), x_{j_2}(t) , \dots  ,x_{j_k}(t)\} \right)} \subset {H(t)}^+\] and \[ {\conv  \left(B(t) \backslash \{x_{j_1}(t), x_{j_2}(t) , \dots  ,x_{j_k}(t)\} \right)} \subset {H(t)}^-.\] So there is a subsequence of the sequence of hyperplanes $\{{H(t_n)} \}$  that converges to a hyperplane $H$, which necessarily weakly separates
\[
{\conv  \left(A(t_0) \backslash \{x_{j_1}, x_{j_2}, \dots  ,x_{j_k}\} \right)} \text{ from } {\conv  \left(B(t_0) \backslash \{x_{j_1}, x_{j_2}, \dots  ,x_{j_k}\} \right)}. 
\]

By hypothesis,
\[ {\conv  \left(A(t_0) \backslash \{x_{j_1}, x_{j_2}, \dots  ,x_{j_k}\} \right)} 
\cap {\conv  \left(B(t_0) \backslash \{x_{j_1}, x_{j_2}, \dots  ,x_{j_k}\} \right)} \neq \emptyset,\] which implies that
\[ \conv  \left(A(t_0) \backslash \{x_{j_1}, x_{j_2}, \dots  ,x_{j_k}\} \right)
\cap  \conv  \left(B(t_0) \backslash \{x_{j_1}, x_{j_2}, \dots  ,x_{j_k}\} \right) \cap H  \neq \emptyset.\]

By Radon's theorem, since the points of $X(0)$ are in general position, the plane $H$ has to contain $d+1$ points of $X(t_0)$, one of which has to be the point $x_1(t_0)$ and none of which are in the set $\{x_{j_1}, x_{j_2}, \dots  ,x_{j_k}\}$, and the points in $X(t_0) \cap H$ can be divided into two sets $A'(t_0)$ and $B'(t_0)$ such that \[ \conv ( A'(t_0) )\cap \conv  ( B'(t_0)) \neq \emptyset . \]  

Consequently, by the induction hypothesis, there are $kd+k+2$ points in general position outside the plane $H$ for which we can find a partition $A''(t_0)$, $B''(t_0)$ such that 
\begin{gather*} 
{conv \left(A''(t_0) \backslash \{x_{i_1}, x_{i_2}, ..., x_{i_{k-1}}\} \right)} 
\cap {conv \left(B''(t_0)  \backslash \{x_{i_1}, x_{i_2}, ..., x_{i_{k-1}}\} \right)} \neq \emptyset\\
 \forall \{x_{i_1}, x_{i_2}, ..., x_{i_{k-1}}\} \subset X(t_0).
 \end{gather*}

Suppose without loss of generality that $x_1(t_0) \in A'(t_0)$. Then for $t_0 \leq t \leq 1$, say $x_1(t) \in H^{+}.$  Now consider the following partition for $X(t)$: 
\[
A(t)= A''(t_0) \cup (A'(t_0) \backslash \{x_1(t_0))\} \cup \{x_1(t)\}, \quad B(t)= B''(t_0) \cup B'(t_0).
\]

For all $X_k =\{x_{i_1}, x_{i_2}, \dots , x_{i_k}\}$, subsets of $X(t)$, if $\left| (X_k \cap H) \cup \{x(t)\} \right|\geq 1$, the lemma holds. So the only case remaining to be dealt with is when  $X_k \subset \{ A''(t_0) \cup B''(t_0) \}$.  \\

Observe that if there is $x_a \in \{ A''(t_0)\backslash X_k \cap H^{-} \}$ or $x_b \in \{ B''(t_0)\backslash X_k \cap H^{+} \}, $ then for some $t=t_0+\epsilon$, 
\begin{equation} \label{eqn: uno}
 \emptyset \neq \conv  \left( A'(t) \cup \{x_a\} \right) \cap \conv  \left( B'(t) \right) \subset 
\conv  \left( A(t) \backslash X_k \right) \cap \conv  \left( B(t) \backslash X_k \right) 
\end{equation}or 
\begin{equation} \label{eqn: dos}
 \emptyset \neq \conv  \left( A'(t)  \right) \cap \conv  \left( B'(t) \cup \{x_b\} \right) \subset 
\conv  \left( A(t) \backslash X_k \right) \cap \conv  \left( B(t) \backslash X_k \right). 
\end{equation}

So we only need to examine the case when there is $X_{k}$ such that $\{ A''(t_0)\backslash X_k \cap H^{-} \}= \emptyset$  and $\{ B''(t_0)\backslash X_k \cap H^{+} \}=\emptyset.$ In such a situation, $B''(t_0)\backslash X_k \subset H^{-}$ and $A''(t_0)\backslash X_k \subset H^{+}$. At least one of 
\[
B''(t_0)\backslash X_k \neq \emptyset \quad \text{or}\quad A''(t_0)\backslash X_k \neq \emptyset
\]
$A''(t_0)$ and $B''(t_0)$ can be swapped in the partition, and one of \eqref{eqn: uno} or \eqref{eqn: dos} will hold. 
This proves that  $\exists \: t>t_{0}$ such that there is a partition $A(t), B(t)$ of $X(t)$ for which for any $X_{k} \subset X$, with $|X_{k}|=k$,
\[\conv (A(t)\backslash X_{k}) \cap \conv (B(t)\backslash X_{k}) \neq \emptyset\] holds, which is a contradiction. Therefore $t_{0}=1$, and the lemma holds.
\end{proof}

Together, Lemmas \ref{LBASE}, \ref{GENERAL} and \ref{RADONSOTE} constitute the proof of Theorem \ref{TPARTITION} and tracing back to equivalences presented in Section \ref{equivalent}, Theorem \ref{PROJECTIVE} has also been proved.

\section{Final remarks}

McMullen's problem was originally posed as a geometrical property of a configuration of points, and even the generalization dealt with in this paper is a geometrical interpretation. However, the partition problem to which it is equivalent is very interesting in itself and does not need to have any restriction on $k$, the number of points removed.  

The upper bound in Theorem \ref{TPARTITION} holds in the general setting even if $k \geq \frac {d}{2}$; in particular the bounds are sharp for $d=2$ and $k=2,3$ \cite{NAT}.

Furthermore, by increasing the number of partitions allowed, the following Tverberg type question arises:

 \emph{Determine the smallest number $\lambda{(d, s, k)}$ such that for any set $X$ of $\lambda{(d, s ,k)}$  points in $R^d$ there exists a subdivision of X into s sets $A_{1},A_{2},\dots ,A_{s}$ such that \[\bigcap_{i = 1}^s  \conv {(A_i \backslash {\{x_1, x_2, \dots , x_k\}})}  \neq \emptyset, \quad \forall \;   {\{x_1, x_2, \dots , x_k\}}  \subset X.\]}

This problem is one of many questions, such as Reay's conjecture, that rather than studying \emph{when} the partitions of the sets intersect, focuses on \emph{how} they intersect.  

Superficially, it seems that the dimension of the intersection of the convex hulls of partitions might bear a relationship to $k$-divisibility. 

The following very loose bound is a direct consequence of Tverberg's theorem.

\begin{lemma} Let $\lambda(d,s,k)$ be the smallest number such that each set $X$ of  $\lambda(d,s,k)$ points in $\mathbb{R}^{d}$ can be divided into $s$ pairwise disjoint sets $A_{1},A_{2},\dots ,A_{s}$ such that $\bigcap_{i = 1}^s  \conv {(A_i \backslash {\{x_1, x_2, \dots , x_k\}})}  \neq \emptyset$, for all subsets ${\{x_1, x_2, \dots , x_k\}}  \subset X.$ Then $\lambda(d,s,k)\geq (k+1)( (s-1)(d-1)+1)$.
\end{lemma}

\nocite{*}
\bibliographystyle{amsplain}
\bibliography{bib_cdm_projective}

\end{document}